\documentclass[reqno, 10pt]{article}
\usepackage[utf8]{inputenc}
\usepackage{mathtools,mathrsfs,amssymb,amsthm,amsmath,bm}
\usepackage{comment}
\usepackage{hyperref}
\usepackage{url}
\usepackage[a4paper, left = 1.5in, right = 1.5in, top =1.5in, bottom = 2in]{geometry}
\usepackage{dsfont}   
\usepackage[british]{babel}

\newtheorem{thm}{Theorem}
\newtheorem{pro}{Proposition}

\theoremstyle{definition}
\newtheorem*{rem}{Remark}

\makeatletter
\renewenvironment{proof}[1][\proofname] {\par\pushQED{\qed}\normalfont\topsep6\p@\@plus6\p@\relax\trivlist\item[\hskip\labelsep\bfseries#1\@addpunct{.}]\ignorespaces}{\popQED\endtrivlist\@endpefalse}
\makeatother

\def\[#1\]{\begin{align*}#1\end{align*}}

\newcommand{\R}{\mathbb{R}}
\newcommand{\N}{\mathbb{N}}
\newcommand{\ceq}{\coloneqq}
\newcommand{\lrtx}[1]{\ \text{#1} \ }

\def\I{\mathds{1}}

\begin{document}
\title{Uniform Approximation of Continuous Functions by Nontrivial Simple  Functions}
\author{Yu-Lin  Chou\thanks{Yu-Lin Chou, Institute of Statistics, National Tsing Hua University, Hsinchu 30013, Taiwan,  R.O.C.; Email: \protect\url{y.l.chou@gapp.nthu.edu.tw}.}}
\date{}
\maketitle

\begin{abstract}
We prove that  every nonnegative  continuous real-valued  function on a given compact  metric space is the uniform limit of some  increasing sequence of nonnegative   simple functions being linear combinations of indicators of open sets;   
here the nontriviality is relative to the standard choice(s) of approximating  simple functions for measurable   functions, where one loses control over the indicated  measurable  sets. Thus the standard uniform approximation of bounded nonnegative  measurable real-valued  functions by  increasing nonnegative   simple functions may be improved for nonnegative continuous real-valued functions on compact metric spaces.
There are also some interesting consequences regarding semi-continuous functions and  smooth functions.\\

\end{abstract}

\section{Introduction}
Let $\R_{+} \ceq \{ x \in \R \mid x \geq  0 \}$.
A continuous   $\R_{+}$-valued function on a compact metric space, 
being Borel and bounded then, 
is the uniform limit of some  increasing sequence of  nonnegative simple functions, 
the standard choice(s) (e.g. Theorem 2.10 in  Folland \cite{fol}) of the approximating simple functions utilizing the preimages of half-open intervals under the given function. Such a classical  approximation is silent on whether or not the involved indicator functions are  indicators of open sets.   

The present short communication furnishes an improvement of a uniform approximation for continuous $\R_{+}$-valued functions on a compact metric space by an increasing sequence of nonnegative simple functions; the involved indicators may be chosen such that they are indicators of open sets.
Since simple real-valued   functions are in general not continuous with respect to the standard topology of $\R$,
the relevant existing results (such as Goodstein \cite{g})  regarding uniform convergence of a monotonic sequence of  continuous functions do not (immediately) apply.
Our approximation  thus also adds value from this aspect. 

Another facet of our result is supplying a new proof of the customary uniform approximation for a significant subclass of Borel functions.

Further,
our main result and its proof  imply some interesting results  regarding semi-continuous functions and smooth functions.

\section{Results}
With respect to the intended purpose,
our main result is the slightly stronger 
\begin{thm}
If $\Omega$ is a compact metric space,
if $f: \Omega \to \R_{+}$ is continuous,
and if $(a_{j})_{j \in \N}$ is a vanishing sequence of reals $a_{j} > 0$  such that $\sum_{j \in \N}a_{j}$ diverges, 
then there are some open subsets $G_{1}, G_{2}, \dots$ of $\Omega$ such that $(\sum_{j=1}^{n}a_{j}\I_{G_{j}})_{n \in \N}$ converges uniformly to $f$.
\end{thm}
\begin{proof}
We claim that there are some open subsets $G_{1}, G_{2}, \dots$ of $\Omega$ such that  
\[
f(x) = \sum_{j \in \N}a_{j}\I_{G_{j}}(x)
\]
for all $x \in \Omega$,
i.e. that $f$ is the pointwise limit of the sequence $(\sum_{j=1}^{n}a_{j}\I_{G_{j}})_{n \in \N}$.
Indeed, 
Theorem 2.3.3 in Federer \cite{fed} gives in particular (with a one-line proof sketch)   a   ``harmonic'' representation of $f$ in the desired form with Borel $G_{j}$.  
Theorem 1.12 in Evans and Gariepy \cite{eg} provides a more detailed proof of the Federer's result for the special case where $(a_{j})$ is the harmonic sequence $(j^{-1})$. 
This proof is readily generalizable to cover our $(a_{j})$, 
and a detailed proof of the generalization, 
 adapted from the  proof, 
may be found in the proof of Theorem 1 in Chou \cite{c}.

To prove the claim,
we modify the proof in Chou \cite{c} as follows.   
If $G_{1} \ceq \{ x \in \Omega \mid f(x) > a_{1} \}$,
let $G_{n} \ceq \{ x \in \Omega \mid f(x) > a_{n} +  \sum_{j=1}^{n-1}a_{j}\I_{G_{j}} \}$ for all $n \geq 2$ 
by induction.
Then each $G_{n}$ is open in $\Omega$. 
It then can be shown that $f \geq  \sum_{j \in \N}a_{j}\I_{G_{j}}$ on $\Omega$; 
the assumed finiteness of $f$ and the assumptions on $(a_{j})$ in turn imply the desired pointwise convergence, which proves the claim.

To obtain the desired uniform convergence,
we remark that each $-a_{j}\I_{G_{j}}$ is upper semi-continuous. 
Moreover, a sum of finitely many upper semi-continuous functions with values in $\R$ is  again upper semi-continuous: 
If $g_{1}, g_{2}$ are real-valued and upper semi-continuous, 
and if $b > 0$,
then for every $x$ (in the domain)  there is some open  neighborhood $V$ of $x$,
e.g. $V$ being the intersection of some open neighborhoods of $x$ that are included respectively  in $\{ y \mid g_{1}(y) < g_{1}(x) +  b/2 \}$ and in 
$\{ y \mid g_{2}(y) < g_{2}(x) +  b/2 \}$,  
such that 
 $x \in V \subset \{ y \mid g_{1}(y) + g_{2}(y) < g_{1}(x) + g_{2}(x) + b  \}.$

Since $f$ is in particular upper semi-continuous by assumption,
the function $f - \sum_{j=1}^{n}a_{j}\I_{G_{j}}$ is upper semi-continuous for every $n \in \N$. 
The previously proved   pointwise convergence claim implies that $f - \sum_{j=1}^{n}a_{j}\I_{G_{j}}$ decreases to  $0$ pointwisely; it then follows from a variant of Dini's theorem (e.g. the theorem in  M8 in  Appendix M of  Billingsley \cite{b}) 
that 
\[
f - \sum_{j=1}^{n}a_{j}\I_{G_{j}}
\to 0 \lrtx{uniformly on} \Omega
\] 
as $n \to \infty$.
\end{proof}

\begin{rem}
That part of  proof of Theorem 1 for uniform convergence    is in some sense a new proof;
since the function $f$ in Theorem 1 is bounded,
the uniform convergence also follows  from the inequality 
$f \geq \sum_{j \in \N}a_{j}\I_{G_{j}}$.
But our new proof exploits further properties of the involved functions,
and hence may be potentially useful for other purposes. \qed 
\end{rem}

Theorem 1 and its proof are  also informative in another respect:
\begin{pro}
Every continuous $\R_{+}$-valued function on a compact metric space is the uniform limit of some increasing sequence of lower semi-continuous $\R_{+}$-valued functions.
\end{pro}  
\begin{proof}
Every indicator function of an open set is lower semi-continuous. 
The sum of finitely many  lower semi-continuous functions with values in $\R$ is lower semi-continuous; indeed, if $g_{1}, g_{2}$ are $\R$-valued and lower semi-continuous, 
and if $b > 0$,
then for every $x$ (in the domain)  there is some open neighborhood $V$ of $x$,  e.g. $V$ being the intersection of some open neighborhoods of $x$ included respectively in 
$\{ y \mid g_{1}(y) > g_{1}(x) - b/2 \}$ and in $\{ y \mid g_{2}(y) > g_{2}(x) - b/2 \}$, 
such that $x \in V \subset \{ y \mid  g_{1}(y) + g_{2}(y) > g_{1}(x) + g_{2}(x) - b \}$.

Let $\sum_{j \in \N}a_{j}\I_{G_{j}}$ be a   uniform representation of $f$ as given by Theorem 1.
Since $f$ is in particular lower semi-continuous by assumption, and since each $a_{j}\I_{G_{j}}$ 
and hence each $\sum_{j=1}^{n}a_{j}\I_{G_{j}}$
is lower semi-continuous, 
the sequence $(\sum_{j=1}^{n}a_{j}\I_{G_{j}})_{n \in \N}$ serves the purpose.
\end{proof}

A part of the proof of Theorem 1 and of that of  Proposition 1 together  suggest a lower bound for nonzero  continuous functions $\R^{n} \to \R_{+}$ in terms of a series of nonzero smooth functions    $\R^{n} \to \R_{+}$:
\begin{pro}
For every nonzero  continuous $f: \R^{n} \to \R_{+}$,
there are some nonzero smooth $g_{1}, g_{2}, \dots : \R^{n} \to \R_{+}$ such that  
i)
$f(x) \geq \sum_{j \in \N}g_{j}(x)$ for all $x \in \R^{n}$
and
ii) $f \geq \sum_{j \in \N}g_{j}$ uniformly on every compact $K \subset \R^{n}$.
\end{pro}
\begin{proof}
The proofs of Theorem 1 and  Proposition 1 imply in particular that $f$ is the pointwise limit of some increasing sequence of $\R_{+}$-valued lower semi-continuous functions; the compactness assumption plays a role precisely in asserting uniform convergence. 

Lemma 1.1.5 in Krantz and Parks  \cite{k} asserts in particular that every nonzero lower semi-continuous function $\R^{n} \to \R_{+}$ may be minorized by some nonzero  smooth function $\R^{n} \to \R_{+}$ (pointwisely).

If we   represent $f$ as  $\sum_{j \in \N}a_{j}\I_{G_{j}}$ on $\R^{n}$ by the proof of Theorem 1,
then the lower semi-continuity of each $a_{j}\I_{G_{j}}$ implies that for every $j \in \N$ there is some nonzero smooth $g_{j}: \R^{n} \to \R_{+}$ such that $g_{j} \leq a_{j}\I_{G_{j}}$ on $\R^{n}$;
therefore, we have 
\[
f(x) \geq \sum_{j \in \N}g_{j}(x)
\]  
for all $x \in \R^{n}$.

The other desired  property now  follows 
either from   the boundedness of a continuous function on a compact set or from the proof of the variant of Dini's theorem cited in the proof of Theorem 1.    
\end{proof}

\end{document}